\documentclass{amsart}

\usepackage{tikz}
\usepackage{xcolor}

\usepackage{latexsym,amsmath,amssymb,epic}
\usepackage{latexsym,amssymb,epic}
\usepackage{amsmath,amsthm}
\usepackage{color}

\usepackage{soul}

\usepackage{hyperref}
\hypersetup{
	colorlinks=true, 
	linktoc=all, 
	linkcolor=blue} 

\newtheorem{theorem}{Theorem}[section]
\newtheorem{lemma}[theorem]{Lemma}

\begin{document}

	\title[Merca's Cubic Partitions]{Elementary Proofs of Infinite Families of Congruences for Merca's Cubic Partitions}

	\author{Robson da Silva}
	\address{Universidade Federal de S\~ao Paulo, S\~ao Jos\'e dos Campos, SP 12247--014, Brazil}
	\email{silva.robson@unifesp.br}

	\author{James A. Sellers}
	\address{Department of Mathematics and Statistics, University of Minnesota Duluth, Duluth, MN 55812, USA}
	\email{jsellers@d.umn.edu}

	\subjclass[2010]{11P83, 05A17}
	
	\keywords{partitions, cubic partitions, congruences, generating functions}
	
	
	\maketitle
	
	
\begin{abstract}
Recently, using modular forms and Smoot's {\tt Mathematica} implementation of Radu's algorithm for proving partition congruences, Merca proved the following two congruences:  For all $n\geq 0,$
\begin{align*}
A(9n+5) & \equiv 0 \pmod{3}, \\
A(27n+26) & \equiv 0 \pmod{3}.
\end{align*}
Here $A(n)$ is closely related to the function which counts the number of {\it cubic partitions}, partitions wherein the even parts are allowed to appear in two different colors.  Indeed, $A(n)$ is defined as the difference between the number of cubic partitions of $n$ into an even
numbers of parts and the number of cubic partitions of $n$ into an odd numbers of
parts.

In this brief note, we provide elementary proofs of these two congruences via classical generating function manipulations.  We then prove two infinite families of non--nested Ramanujan--like congruences modulo 3 satisfied by $A(n)$ wherein Merca's original two congruences serve as the initial members of each family.  
\end{abstract}

\section{Introduction}  
In a recent work, Merca \cite{Merca} extensively studied the function which he called $A(n)$ which is defined to be the difference between the number of cubic partitions of $n$ into an even
numbers of parts and the number of cubic partitions of $n$ into an odd numbers of
parts.  (Cubic partitions were introduced by Chan \cite{Chan1, Chan2} to be integer partitions in which even parts are allowed to appear in two different colors.  Chen introduced these in connection with Ramanujan’s cubic continued fraction.)  
Merca notes that the generating function for $A(n)$ is given by 
\begin{equation*}
    \sum_{n=0}^{\infty} A(n)q^n  = (q;q^2)_{\infty}(q^2;q^4)_{\infty} = \frac{f_1}{f_4}
\end{equation*}
where the $q$--Pochhammer symbol $(a;q)_\infty$ is defined by 
$$
(a;q)_\infty = (1-a)(1-aq)(1-aq^2)(1-aq^3)\dots
$$
and $f_a^b$ is defined by 
$$
f_a^b = (q^a;q^a)_\infty^b
$$
for positive integers $a$ and any integer $b.$

In \cite{Merca}, Merca proved the following two Ramanujan--like congruences satisfied by $A(n):$
\begin{theorem}[{\cite[Theorem 1.10]{Merca}}] 
\label{Merca1}
For all $n \geq 0$,
\begin{align}
A(9n+5) & \equiv 0 \pmod{3}, \label{m1} \\
A(27n+26) & \equiv 0 \pmod{3}.\label{m2}
\end{align}
\label{T1}
\end{theorem}
Merca's proof of these two congruences relies solely on Smoot's {\tt Mathematica} implementation of Radu's algorithm for proving partition congruences (which relies heavily on the machinery of modular forms).  Indeed, Merca's proof of (\ref{m1}) involves finding a generating function for $A(9n+5)$ which is produced by Smoot's {\tt Mathematica} package. In this case, the generating function in question turns out to be a non--trivial linear combination of four ratios of eta products (in essence, ratios of products whose terms involve $q$--Pochhammer symbols as defined above), while the generating function for $A(27n+26)$ contains a dozen such terms.  

Our goal in the work below is to provide elementary proofs of Merca's two congruences via classical generating function manipulations and dissections.  We then significantly extend Merca's work on such divisibility properties for $A(n)$ by proving two infinite families of non--nested Ramanujan--like congruences modulo 3 satisfied by $A(n)$ wherein Merca's two congruences above serve as the initial members of each family. Indeed, we will prove the following: 

For all $j\geq 0$ and all $n\geq 0,$
$$
A\left(9^{j+1}n + \frac{39\cdot9^j+1}{8}\right) \equiv 0 \pmod{3}
$$
and 
$$
A\left(3\cdot 9^{j+1}n + \frac{23\cdot9^{j+1}+1}{8}\right) \equiv 0 \pmod{3}.
$$
In order to accomplish the above goals, we require a few classical tools.  First, we recall Ramanujan's theta functions
\begin{align}
f(a,b) & :=\sum_{n=-\infty}^\infty a^\frac{n(n+1)}{2}b^\frac{n(n-1)}{2}, \mbox{ for } |ab|<1, \nonumber \\
\phi(q) & := f(q,q) = \sum_{n=-\infty}^{\infty} q^{n^2} = \frac{(q^2;q^2)_{\infty}^5}{(q;q)_{\infty}^2(q^4;q^4)_{\infty}^2}, 
\nonumber
\\
\psi(q) & := f(q,q^3) = \sum_{n=0}^{\infty} q^{n(n+1)/2} = \frac{(q^2;q^2)_{\infty}^2}{(q;q)_{\infty}}. 
\nonumber
\end{align}
These functions satisfy many interesting properties (see Entries 18, 19, and 22 in \cite{Berndt}), including:
\begin{align}
\phi(-q) & = \frac{(q;q)_{\infty}^2}{(q^2;q^2)_{\infty}}, 
\nonumber 
\\
\psi(-q) & = \frac{(q;q)_{\infty}(q^4;q^4)_{\infty}}{(q^2;q^2)_{\infty}}. 
\nonumber
\end{align}

As noted above, the generating function for $A(n)$ is given by 
\begin{align}
    \sum_{n=0}^{\infty} A(n)q^n & = (q;q^2)_{\infty}(q^2;q^4)_{\infty} \nonumber \\
    & = \frac{f_1}{f_4} \nonumber \\
    & = \frac{f_1^2}{f_2}\frac{f_2}{f_1f_4} \nonumber \\
    & = \frac{\phi(-q)}{\psi(-q)}. \label{gf1}
\end{align}

In order to prove the congruences in question, we require a few well--known $q$--series dissections.

\begin{lemma}
We have
\begin{equation*}
    \phi(-q) = \frac{f_9^2}{f_{18}} -2q\frac{f_3f_{18}^2}{f_6f_9}. 
\end{equation*}
\label{lemma1}
\end{lemma}

\begin{proof}
A proof of this identity can be seen in \cite[Eq. 14.3.4]{H}.
\end{proof}

\begin{lemma}
We have
\begin{equation*}
\frac{1}{\psi(-q)} = \frac{f_{18}^9}{f_3^2f_9^3f_{12}^2f_{36}^3} + q\frac{f_6^2f_{18}^3}{f_3^3f_{12}^3} + q^2\frac{f_6^4f_9^3f_{36}^3}{f_3^4f_{12}^4f_{18}^3}. 
\end{equation*}
\label{lemma2}
\end{lemma}

\begin{proof}
A proof of this identity appears in \cite[Lemma 2.1]{Toh}.
\end{proof}

\begin{lemma}
We have
\begin{equation*}
\frac{1}{\phi(-q)} = \frac{f_6^4f_{9}^6}{f_3^8f_{18}^3} + 2q\frac{f_6^3f_{9}^3}{f_3^7} + 4q^2\frac{f_6^2f_{18}^3}{f_3^6}. 
\end{equation*}
\label{lemma3}
\end{lemma}

\begin{proof}
A proof of this result can be seen in \cite[Theorem 1]{HS}.
\end{proof}

\begin{lemma}
We have
\begin{equation*}
\frac{1}{\psi(q)} = \frac{f_3^2f_{9}^3}{f_6^6} - q\frac{f_3^3f_{18}^3}{f_6^7} + q^2\frac{f_3^4f_{18}^6}{f_6^8f_9^3}. 
\end{equation*}
\label{lemma4}
\end{lemma}

\begin{proof}
A proof of this result can be seen in \cite[Lemma 2.2]{HS1}.
\end{proof}

These are all of the tools that we need in order to prove Merca's congruences in an elementary fashion.  We now transition to providing these proofs.  

\section{Elementary proof of Theorem \ref{T1}}

Initially, we use Lemmas \ref{lemma1} and \ref{lemma2} to extract the terms involving $q^{3n+2}$ in (\ref{gf1}):
\begin{equation*}
\sum_{n=0}^{\infty} A(3n+2)q^{3n+2} = q^2 \frac{f_6^4f_9^5f_{36}^3}{f_3^4f_{12}^4f_{18}^4} - 2q^2\frac{f_6f_{18}^5}{f_3^2f_9f_{12}^3}.
\end{equation*}
Dividing by $q^2$ and replacing $q^3$ by $q$ yields
\begin{align}
\sum_{n=0}^{\infty} A(3n+2)q^{n} & = \frac{f_2^4f_3^5f_{12}^3}{f_1^4f_{4}^4f_{6}^4} - 2\frac{f_2f_{6}^5}{f_1^2f_3f_{4}^3} \nonumber \\
& \equiv \frac{f_2}{f_1f_4} \frac{f_3^4f_{12}^2}{f_6^3} - 2\frac{f_2}{f_1^2} \frac{f_6^5}{f_3f_{12}} \pmod{3}. \label{eq2.1}
\end{align}
Using Lemmas \ref{lemma2} and \ref{lemma3} we extract the terms of the form $q^{3n+1}$ from both sides of the last congruence to obtain
\begin{align*}
\sum_{n=0}^{\infty} A(9n+5)q^{3n+1} 
& \equiv q\frac{f_3f_{18}^3}{f_6f_{12}} - 4q\frac{f_6^8f_9^3}{f_3^8f_{12}} \pmod{3}\\
& \equiv  q\frac{f_3f_{18}^3}{f_6f_{12}} - 4q\frac{f_3}{f_6f_{12}} \frac{f_6^9f_9^3}{f_3^9} \pmod{3}\\
& \equiv -3q\frac{f_3f_{18}^3}{f_6f_{12}} \pmod{3} \\
& \equiv 0 \pmod{3},
\end{align*}
which proves \eqref{m1}.

Using Lemmas \ref{lemma2} and \ref{lemma3} we extract the terms of the form $q^{3n+2}$ from both sides of \eqref{eq2.1}:
\begin{align*}
\sum_{n=0}^{\infty} A(9n+8)q^{3n+2} 
& \equiv  q^2\frac{f_6f_{9}^3f_{36}^3}{f_{12}^2f_{18}^3} - 8q^2\frac{f_6^7f_{18}^3}{f_3^7f_{12}} \pmod{3}.
\end{align*}
Dividing by $q^2$ and replacing $q^3$ by $q$, we obtain
\begin{align}
\sum_{n=0}^{\infty} A(9n+8)q^{n} 
& \equiv  \frac{f_2f_3^3f_{12}^3}{f_{4}^2f_{6}^3} - 8\frac{f_2^7f_{6}^3}{f_1^7f_{4}} \pmod{3} \nonumber \\
& \equiv \frac{f_2}{f_4^2} \frac{f_3^3f_{12}^3}{f_6^3} -8\frac{f_2}{f_1f_4} \frac{f_6^5}{f_3^2} \pmod{3}. \label{eq2.2}
\end{align}
Now we employ Lemmas \ref{lemma2} and \ref{lemma4} to extract the terms of the form $q^{3n+2}$ from the congruence above. The resulting congruence after division by $q^2$ and replacing $q^3$ by $q$ is given by
\begin{align*}
\sum_{n=0}^{\infty} A(27n+26)q^{n} 
& \equiv  -\frac{f_1^3f_{12}^3}{f_{4}^4} - 8\frac{f_2^9f_3^3f_{12}^3}{f_1^6f_{4}^4f_6^3} \pmod{3}\\
& \equiv -\frac{f_1^3f_{12}^3}{f_4^4}  -8\frac{f_1^3f_{12}^3}{f_4^4} \pmod{3}  \\
& \equiv 0 \pmod{3},
\end{align*}
which proves \eqref{m2}.


\section{Infinite families of congruences modulo 3}
While the elementary proofs of Merca's original congruences that we have provided above are very satisfying, it turns out that much more is true about $A(n)$ modulo 3, and our elementary approach to proving these congruences yields the insights needed in order to see this.  In this light, we now proceed to proving two infinite families of non--nested Ramanujan--like congruences modulo 3 satisfied by $A(n).$

We begin by proving the following internal congruence satisfied by $A(n)$ which will serve as an important component in our remaining proofs.

\begin{theorem}
For all $n \geq 0$,
\begin{equation}
    A(27n+8) \equiv A(3n+1) \pmod{3}.
    \label{m4}
\end{equation}
\end{theorem}

\begin{proof}
We prove this result by showing that both $A(27n+8)$ and $A(3n+1)$ have the same generating function modulo 3. Using Lemmas \ref{lemma2} and \ref{lemma4} we extract the terms of the form $q^{3n}$ from \eqref{eq2.2}:
\begin{align*}
\sum_{n=0}^{\infty} A(27n+8)q^{3n} 
& \equiv \frac{f_3^3f_{18}^3}{f_6f_{12}^3}  -8\frac{f_6^5f_{18}^9}{f_3^4f_9^3f_{12}^2f_{36}^3} \pmod{3},
\end{align*}
which after replacement of $q^3$ by $q$ yields
\begin{align}
\sum_{n=0}^{\infty} A(27n+8)q^{n} 
& \equiv \frac{f_1^3f_{6}^3}{f_2f_{4}^3}  -2\frac{f_2^5f_{6}^9}{f_1^4f_3^3f_{4}^2f_{12}^3} \pmod{3}.
\label{eq2.3}
\end{align}

The generating function for $A(3n+1)$ is obtained in the same way using \eqref{gf1} and Lemmas \ref{lemma1} and \ref{lemma2}:
\begin{align*}
\sum_{n=0}^{\infty} A(3n+1)q^{3n+1} 
& = q\frac{f_6^2f_9^2f_{18}^2}{f_3^3f_{12}^3}  -2q\frac{f_{18}^{11}}{f_3f_6f_9^4f_{12}^2f_{36}^3}.
\end{align*}
Dividing this expression by $q$ and replacing $q^3$ by $q$, we are left with
\begin{align*}
\sum_{n=0}^{\infty} A(3n+1)q^{n} 
& = \frac{f_2^2f_3^2f_{6}^2}{f_1^3f_{4}^3}  -2 \frac{f_{6}^{11}}{f_1f_2f_3^4f_{4}^2f_{12}^3} \\
& \equiv \frac{f_1^3f_2^3f_{6}^3}{f_2f_4^3f_6}  -2 \frac{f_1^3f_2^5f_{6}^{11}}{f_1^4f_2^6f_3^4f_{4}^2f_{12}^3} \pmod{3} \\ 
& \equiv \frac{f_1^3f_{6}^3}{f_2f_{4}^3}  -2 \frac{f_2^5f_{6}^{9}}{f_1^4f_3^3f_{4}^2f_{12}^3} \pmod{3},
\end{align*}
which coincides with \eqref{eq2.3} and the proof is complete.
\end{proof}

We can now prove the following two theorems.  

\begin{theorem}
\label{infinite_family1}
For all $j\geq 0$ and all $n\geq 0,$
$$
A\left(9^{j+1}n + \frac{39\cdot9^j+1}{8}\right) \equiv 0 \pmod{3}.
$$
\end{theorem}
\begin{proof}
This theorem follows from a straightforward proof by induction on $j.$  
First, we note that the basis case, $j=0,$ is simply the statement that, for all $n\geq 0,$ 
$$
A(9n+5) \equiv 0 \pmod{3}.
$$
This is Theorem \ref{Merca1} equation (\ref{m1}) above, the first of Merca's original congruences.  

Next, we assume that the statement is true for some $j\geq 0.$  We then wish to prove that 
$$
A\left(9^{j+2}n + \frac{39\cdot9^{j+1}+1}{8}\right) \equiv 0 \pmod{3}
$$
for all $n\geq 0$ as well.  Note that 
\begin{align*}
9^{j+2}n + \frac{39\cdot9^{j+1}+1}{8}
&=
27\left(3\cdot 9^jn+\frac{13\cdot 9^j}{8} - \frac{7}{24}\right) + 8.
\end{align*}
Thanks to (\ref{m4}), we know that, for all $n\geq 0,$
$$A(27n+8)\equiv A(3n+1) \pmod{3}.$$ 
Thus, 
\begin{align*}
A\left(9^{j+2}n + \frac{39\cdot9^{j+1}+1}{8}\right)
&=
A\left(27\left(3\cdot 9^jn+\frac{13\cdot 9^j}{8} - \frac{7}{24}\right) + 8\right) \\
& \equiv 
A\left(3\left(3\cdot 9^jn+\frac{13\cdot 9^j}{8} - \frac{7}{24}\right) + 1\right) \pmod{3}\\
&= 
A\left( 9^{j+1}n+\frac{39\cdot 9^{j}+1}{8} \right) \\
&\equiv 0 \pmod{3} 
\end{align*}
by the induction hypothesis.  The result follows.  
\end{proof}

\begin{theorem}
\label{infinite_family2}
For all $j\geq 0$ and all $n\geq 0,$
$$
A\left(3\cdot 9^{j+1}n + \frac{23\cdot9^{j+1}+1}{8}\right) \equiv 0 \pmod{3}.
$$
\end{theorem}
\begin{proof}
As above, this theorem follows from a straightforward proof by induction on $j.$ Note that  the basis case, $j=0,$ is simply the statement that, for all $n\geq 0,$ 
$$
A(27n+26) \equiv 0 \pmod{3}.
$$
This is Theorem \ref{Merca1} equation (\ref{m2}) above, the second of Merca's original congruences. 

Next, we assume that the statement is true for some $j\geq 0.$  We then wish to prove that 
$$
A\left(3\cdot 9^{j+2}n + \frac{23\cdot9^{j+2}+1}{8}\right) \equiv 0 \pmod{3}
$$
for all $n\geq 0$ as well.  Note that 
\begin{align*}
3\cdot 9^{j+2}n + \frac{23\cdot9^{j+2}+1}{8}
&=
27\left(9^{j+1}n+\frac{3(23)\cdot 9^j}{8} - \frac{7}{24}\right) + 8.
\end{align*}
Thanks to (\ref{m4}), we know that, for all $n\geq 0,$
$$A(27n+8)\equiv A(3n+1) \pmod{3}.$$ 
Thus, 
\begin{align*}
A\left(3\cdot 9^{j+2}n + \frac{23\cdot9^{j+2}+1}{8}\right)
&=
A\left(27\left(9^{j+1}n+\frac{3(23)\cdot 9^j}{8} - \frac{7}{24}\right) + 8\right) \\
& \equiv 
A\left(3\left(9^{j+1}n+\frac{3(23)\cdot 9^j}{8} - \frac{7}{24}\right) + 1\right) \pmod{3}\\
&= 
A\left( 3\cdot 9^{j+1}n + \frac{23\cdot9^{j+1}+1}{8} \right) \\
&\equiv 0 \pmod{3} 
\end{align*}
by the induction hypothesis.  The result follows.  

\end{proof}



\end{document}